\documentclass[12pt]{article}
\usepackage{amsfonts}
\usepackage{setspace}
\usepackage{pgf,tikz}
\usepackage{fullpage}
\usepackage{anysize}
\usepackage{url}
\usepackage{hyperref}
\usepackage{geometry}
\usepackage{amssymb}
\usepackage{latexsym}
\usepackage{amsmath}
\usepackage{mathrsfs}
\usepackage{amsthm, amssymb}

\usepackage{abstract}
\usepackage{amsmath}
\usepackage{lmodern}

\newtheorem{theorem}{Theorem}[section]

\newtheorem{lemma}[theorem]{Lemma}
\newtheorem{proposition}[theorem]{Proposition}

\makeatletter \renewenvironment{proof}[1][\proofname] {\par\pushQED{\qed}\normalfont\topsep6\p@\@plus6\p@\relax\trivlist\item[\hskip\labelsep\bfseries#1\@addpunct{.}]\ignorespaces}{\popQED\endtrivlist\@endpefalse} \makeatother

\newcommand{\diam}{\mbox{\rm diam}}              %%diameter
\newcommand{\dist}{\mbox{\rm dist}}              %%distance
\newcommand{\loc}{\mbox{\rm loc}}

\begin{document}
\title{Local Hausdorff Measure}
\author{John Dever \\
School of Mathematics\\
Georgia Institute of Technology\\
Atlanta GA 30332-0160}

\maketitle
\begin{abstract}
A local Hausdorff dimension is defined on a metric space. We study its properties and use it to define a local Hausdorff measure. We show that in the case that in the local Hausdorff measure is finite we can recover the global Hausdorff dimension from the local one. Lastly, for a variable Ahlfors Q-regular measure on a compact metric space, we show the Ahlfors regular measure is strongly equivalent to the local Hausdorff measure and that the function $Q$ is equal to the local Hausdorff dimension.
\end{abstract}


\maketitle

\section{Introduction}
Many of the central ingredients in this paper have long been known. The general Carath\'{e}odory construction of metric measures may be found in Federer \cite{Fed}.  The definition of local dimension used in this paper may be found in \cite{Loc}. A definition of a variable Hasdorff measure, akin to the $\lambda^Q$ defined here, may be found in \cite{Sob}. Often local dimension is defined through a measure, what we here call the local dimension of a measure. The latter use is often seen in connection with Multifractal Analysis\cite{Tech}. The equality of $Q$ with the Hausdorff dimension in an Ahlfors $Q$-regular space with constant $Q$ is well known\cite{Hein}. However, neither the equality of a variable $Q$ with the intrinsically defined local Hausdorff measure, nor the strong equivalence of measures, in the case of a compact space, between the local Hausdorff measure and a variable Ahlfors $Q$-regular measure, seem to be known.  The result seems to be of interest since it provides a concept of local dimension that can be defined for any metric space that agrees with the the concept of local dimension often used in multifractal analysis, in the often restrictive case that the latter exists. Moreover, the strong equivalence, in the sense defined below, of the local Hausdorff measure to the variable Ahlfors regular measures provides a concrete realization of such ``multifractal" measures.

The layout of this paper is as follows. We introduce notation and present background information in the remainder of the introduction. We define and investigate properties of the local Hausdorff dimension and local Hausdorff measure and an equivalent local open spherical measure in the next section. Lastly we focus on variable Ahlfors $Q$-regular measures in the case of a compact metric space and connect $Q$ to the local Hausdorff dimension and the $Q$-regular measure to the local Hausdorff measure. 

For $(X,\rho)$ a metric space, we denote the open ball of radius $0\leq r\leq \infty$ about $x\in X$ by $B_r(x)$. We denote the closed ball of radius $r$ by $B_r[x].$
For $A\subset X$ we let $|A|:=\diam(A)=\sup_{[0,\infty]}\{\rho(x,y)\;|\;x,y\in A\}.$ Throughout the paper, let $\mathscr{C}:=\mathscr{P}(X),$ and let $\mathscr{B}$ be the collection of all open balls in $X,$ where $\emptyset = B_0(x)$ and $X=B_\infty(x)$ for 
any $x\in X.$ By a covering class we mean a collection $\mathscr{A}\subset \mathscr{P}(X)$ with $\emptyset, X \in \mathscr{A}.$ We will primarily work with the covering classes $\mathscr{C}=\mathscr{P}(X)$ and $\mathscr{B}.$ For $\mathscr{A}$ a covering class, let $\mathscr{A}_\delta := \{\mathscr{U} \subset \mathscr{A} \;|\;$\mbox{$\mathscr{U}$ at 
countable},\;$ |U|\leq \delta$ \mbox{for}\;$ U\in \mathscr{U},\; A\subset \cup 
\mathscr{U}\}.$ 

Recall an outer measure on a set $X$ is a function $\mu^*:\mathscr{P}(X)\rightarrow [0,\infty]$ such that 
$\mu^*(\emptyset)=0$; if $A,B\subset X$ with $A\subset B$ then $\mu^*(A)\leq \mu^*(B);$ and if 
$(A_i)_{i=1}^\infty \subset \mathscr{P}(X)$ then 
$\mu^*(\cup_{i=1}^\infty) A_i)\leq \sum_{i=1}^\infty 
\mu^*(A_i).$ The second condition is called monotonicity, and the last condition is called countable subadditivity.

A measure $\mu$ on a $\sigma$-algebra $\mathscr{M}$ is called complete if for all $N\in \mathscr{M}$ with $\mu(N)=0,$ $\mathscr{P}(N)\subset \mathscr{M}$.

The following theorem is basic to the subject. See \cite{Folland} for a proof.
\begin{theorem} (Carath\'{e}odory) If $\mu^*$ is an outer measure on $X$ then if $\mathscr{M^*}:=\{A\subset X\;|\;\forall E\subset X\;[\;\mu^*(E)=\mu^*(E\cap A)+\mu^*(E\cap A^c)\;]\;\},$ $\mathscr{M^*}$ is a $\sigma$-algebra and $\mu^*|_{\mathscr{M^*}}$ is a complete measure. 
\end{theorem}

Now suppose $(X,d)$ is a metric space. Sets $A,B\subset X$ are called positively separated if $\dist(A,B)=\inf\{d(x,y)\;|\;x\in A,y\in B\;\}>0.$ An outer measure $\mu^*$ on $X$ is called a metric outer measure if for all $A,B\subset X$ with $A,B$ positively separated, $\mu^*(A\cup B)=\mu^*(A)+\mu^*(B).$ 

The Borel sigma algebra is the smallest $\sigma$-algebra containing the open sets of $X.$ Elements of the Borel $\sigma$-algebra are called Borel sets. A Borel measure is a measure defined on the $\sigma$-algebra of Borel sets. The following proposition is well known. See \cite{Falc1} for a proof.

\begin{proposition} 
If $\mu^*$ is a metric outer measure on a metric space $X,$ then $\mathscr{M}^*$ contains the $\sigma$-algebra of Borel sets. In particular, $\mu^*$ may be restricted to a Borel measure. 
\end{proposition}

For $\tau:\mathscr{C}\rightarrow [0,\infty]$ with $\tau(\emptyset)=0,$
let $\mu^* _{\phi, \delta} (A):=\inf \{ \sum_{U\in \mathscr{U}}\tau(U)\;|\; 
\mathscr{U}\in \mathscr{C}_\delta(A)\}$ and $\mu^* _{\tau}(A)=\sup_{\delta>0} \mu^* _{\tau, \delta} 
(A).$ 
The verification of the following proposition is straightforward. 
\begin{proposition} $\mu^* _{\tau}$ is a metric outer measure. \end{proposition}

Note we may restrict to any covering class $\mathscr{A}$ containing $\emptyset$ and $X$ by setting $\tau(U)=\infty$ for $U\in \mathscr{C}\setminus \mathscr{A}.$

It then follows that the $\mu^* _{\tau}$ measurable sets contain the Borel sigma algebra. Let $\mu_\tau$ be the restriction of $\mu^* _{\tau}$ to the Borel sigma algebra. Then by Caratheodory's Theorem, $\mu_\tau$ is a Borel measure on $X.$ 

For $s\geq 0$ let $H^s$ be the measure obtained from the choice $\tau(U)=|U|^s$ for $U\neq \emptyset$ and $\tau(\emptyset)=0.$ $H^s$ is called the $s$-dimensional Hausdorff measure. Let $\lambda^s$ be the measure obtained by restricting $\tau$ to the smaller covering class $\mathscr{B}$ of open balls. Concretely, let $\lambda^s$ be the measure obtained by setting $\tau(B)=|B|^s$ for $B$ a non-empty open ball, $\tau(\emptyset)=0,$ and $\tau(U)=\infty$ otherwise. We call $\lambda^s$ the $s$-dimensional open spherical measure. 

We call Borel measures $\mu,\nu$ on $X$ strongly equivalent, written $\mu \simeq \nu$, if there exists a constant $C>0$ such that for every Borel set $E$, $\frac{1}{C}\nu(E)\leq \mu(E)\leq C\nu(E).$ 

The proof of the following lemma, while straightforward, is included for completeness.
\begin{lemma} For any $s\geq 0,$ $\lambda^s \simeq H^s.$ 
\end{lemma}
\begin{proof} It is clear that $H^s \leq \lambda^s.$ Conversely, let $A$ be a Borel set. We may assume $A\neq \emptyset.$ Also, we may assume $H^s(A)<\infty,$ since 
otherwise the reverse inequality is clear. If $s=0$ then $H^0$ is the counting measure. So let $H^0(A)=n<\infty.$ Let $x_1,...,x_n$ be an enumeration of the elements of $A.$ Let $r$ be the minimum distance between distinct elements of $A.$ For $0<\delta<r$ let $B_i = 
B_{\frac{\delta}{2}}(x_i)$ for each $i.$ Then $(B_i)_{i=1}^n \in \mathscr{B}_\delta(A)$ 
and $\lambda^{0,*}_\delta(A)\leq n=H^0(A)$. Hence $\lambda^0(A) \leq H^0(A).$ So we may 
assume $s>0.$  Let $\epsilon>0.$ Let $\delta>0$ and let $\mathscr{U}\in 
\mathscr{C}_\delta(A)$ with $\sum_{U\in \mathscr{U}} |U|^s <\infty.$ We may assume $U\neq 
\emptyset$ for each $U\in \mathscr{U}.$ Choose $x_U\in U$ for each $U\in \mathscr{U}.$ 
Let $\mathscr{U}_0=\{U\in \mathscr{U}\;|\;|U|=0\}, \mathscr{U}_1:=\{U\in \mathscr{U}\;|
\;|U|>0\}.$ Then for each $U\in \mathscr{U}_0$ choose $0<r_U<\delta$ such that 
$\sum_{U\in \mathscr{U}_0} r_U^s<\frac{\epsilon}{2^s}.$ For $U\in \mathscr{U}_1$ let $r_U = 2|U|.$ Then let $B_U:=B_{r_U}(x_U)$ for $U\in \mathscr{U}.$ It follows that 
$(B_U)_{U\in \mathscr{U}}\in \mathscr{B}_{4\delta}(A)$ and $\lambda^{s,*}_{4\delta}
(A)\leq \sum_{U\in \mathscr{U}}|B_U|^s \leq 2^s(\sum_{U\in \mathscr{U}_0} r_U^s + 
\sum_{U\in \mathscr{U}_1} r_U^s) \leq \epsilon + 4^s\sum_{U\in \mathscr{U}} |U|^s.$ Hence 
$\lambda^s(A)\leq 4^sH^s(A).$ \end{proof}

Let $X$ be a metric space and $A\subset X$.  Let $0\leq t<s.$ Then if $(U_i)_{i=1}^\infty \in \mathscr{C}_\delta(A)$ then $H^s_\delta(A)\leq \sum_i |U_i|^s \leq \delta^{s-t}\sum_i |U_i|^t.$ So $H^s_\delta(A) \leq \delta^{s-t} H_\delta^t(A)$ for all $\delta>0$.

Suppose $H^t(A)<\infty$. Then since $\delta^{t-s}\overrightarrow{_{_{\delta \rightarrow 0^+}}} 0,$ $H^s(A)=0.$ Similarly, if $H^s(A)>0$ and $t<s,$ $H^t(A)=\infty.$ It follows $\sup\{s\geq 0\;|\;H^s(X)=\infty\}=\inf\{s\geq 0\;|\;H^s(X)=0\}.$ We denote the common number in $[0,\infty]$ by $\dim(A).$ It is called the Hausdorff dimension of $X.$ Since $H^s\simeq \lambda^s,$ we also have $\dim(A)= \sup\{s\geq 0\;|\;\lambda^s(A)=\infty\}=\inf\{s\geq 0\;|\;\lambda^s(A)=0\}.$

\section{Local dimension and measure}
\begin{lemma} If $A\subset B \subset X$ then $\dim(A)\leq \dim(B).$\end{lemma}
\begin{proof} By monotonicity of measure, $H^s(A)\leq H^s(B)$. So $H^s(B)=0$ implies 
$H^s(A)=0.$ Therefore $\dim(A)=\inf\{s\geq 0\;|\;H^s(A)=0\}\leq \inf\{s\geq 0\;|
\;H^s(B)=0\}=\dim(B).$\end{proof} 

Let $\mathscr{O}(X)$ be the collection of open subsets of $X.$ For $x\in X,$ let $\mathscr{N}(x)$ be the open neighborhoods of $x.$ Define $\dim_{\loc}:X\rightarrow [0,\infty]$ by $\dim_{\loc}(x):=\inf \{\dim(U)\;|\; U\in 
\mathscr{N}(x)\}.$  By Lemma 2.1, $\dim_{\loc}(x)=\inf\{\dim(B_\epsilon(x))\; | \;\epsilon>0\}.$ 

\begin{theorem} $\dim_{\loc}$ is upper semicontinuous. In particular, it is Borel measurable. 
\end{theorem}
\begin{proof} Let $c\geq 0.$ If $c=0$ then $\dim_{\loc}^{-1}([0,c))=\emptyset\in\mathscr{O}
(X)$. So let $c>0.$ Then suppose $x \in \dim_{\loc}^{-1}([0,c)).$ Then there exists a $U\in 
\mathscr{N}(x)$ such that $\dim(U)<c.$ Then for $y\in U$, since $U$ is open there exists a $V\in \mathscr{N}(y)$ with $V\subset U.$ So $\dim_{\loc}(y)\leq \dim(V)\leq \dim(U)<c.$ Hence $x\in U \subset \dim_{\loc}^{-1}([0,c)).$ Therefore $\dim_{\loc}^{-1}([0,c))$ is open.
 
\end{proof}

\begin{lemma} If $A$ is a Borel set and $0\leq s_1\leq s_2$ then $\lambda^{s_2}(A)\leq \lambda^{s_1}(A).$ 
\end{lemma}
\begin{proof} If $\mathscr{U}\in \mathscr{B}_\delta(A)$ and $U\in \mathscr{U}$ then diam$(U)^{s_2}\leq \delta^{s_2-s_1}$diam$(U)^{s_1}$. Hence diam$(U)^{s_2}\leq $diam$(U)^{s_1}$ for $0<\delta<1.$ The result follows.
\end{proof}

For $U\subset X, U\neq \emptyset,$ let $\tau(U)=|U|^{\dim(U)}.$ Set $\tau(\emptyset)=0.$ Then $H_{\loc}:=\mu_\tau$ is called the local Hausdorff measure. If we restrict $\tau$ to $\mathscr{B}$ then the measure $\lambda_{\loc}:=\mu_\tau$ is called the local open spherical measure. 

\begin{lemma} If $\dim(X)<\infty$ then $H_{\loc} \simeq \lambda_{\loc}.$\end{lemma}
\begin{proof} Clearly $H_{\loc} \leq \lambda_{\loc}.$ Let $A$ be a Borel set. We may assume $H_{\loc}(A)<\infty$ and $A\neq \emptyset.$ Let $\epsilon>0,0<\delta<\frac{1}{4},$ $\mathscr{U}\in \mathscr{C}_\delta$ with $U\neq \emptyset$ for $U\in \mathscr{U}$ and 
$\sum_{U\in \mathscr{U}}|U|^{\dim(U)}<\infty.$ Let $x_U\in U$ for each $U\in \mathscr{U}.
$ If $|U|=0$ then $U$ is a singleton and so $\dim(U)=0.$ Hence there are at most finitely 
many $U\in \mathscr{U}$ with $|U|=0.$ Let $\mathscr{U}_0$ be the collection of such 
$U\in\mathscr{U}.$ Let $m:=\min_{U\in \mathscr{U}_0}\dim_{\loc}(x_U).$ let $0<r<\delta$ 
such that $(2r)^m\leq 1.$ Then, for $U\in \mathscr{U}_0,$ $U\subset B_r(x_U)$ and, since 
$0<2r<\delta<1$ and $m\leq \dim_{\loc}(x_U)\leq \dim(B_r(x_U)),$ $|B_r(x_U)|
^{\dim(B_r(x_U)}\leq (2r)^m \leq 1.$ For $U\in \mathscr{U}_0$ set $r_U:=r$. Let 
$\mathscr{U}_1$ be the collection of $U\in \mathscr{U}$ with $|U|>0.$ For $U\in 
\mathscr{U}_1$ let $r_U:=2|U|.$ Then for $U\in \mathscr{U}$ let $B_U:=B_{r_U}(x_U).$ Then 
$(B_U)_{U\in \mathscr{U}} \in \mathscr{B}_{4\delta}(A)$ and, since $4|U|\leq 4\delta<1$ 
and $\dim(U)\leq \dim(B_U) \leq \dim(X)<\infty$ for $U \in \mathscr{U},$ $\sum_{U\in 
\mathscr{U}} |B_U|^{\dim(B_U)} \leq \sum_{U\in \mathscr{U}_0} 1 + 
\sum_{U \in \mathscr{U}_1} (4|U|)^{\dim(U)} \leq 4^{\dim(X)} \sum_{U \in \mathscr{U}}|U|^{\dim(U)}.$ 
Hence $\lambda_{\loc} \leq 4^{\dim(X)}H_{\loc}.$
\end{proof}

\begin{proposition}
If $d_0$ is the dimension of $X,$ then $H^{d_0}\ll  H_{\loc}.$ 
\end{proposition}
\begin{proof} 
Suppose $N\subset X$ is Borel measurable with $H_{\loc}(N)=0.$ Let $\epsilon>0.$ Then for all $\delta>0$ there exists a $\mathscr{U}_\delta \in \mathscr{C}_\delta(N)$ such that $\sum_{U\in \mathscr{U}_\delta}|U|^{\dim(U)}<\epsilon.$ But since for  $U\in \mathscr{U}_\delta,$ $\{U\}\in \mathscr{C}_\delta(U),$ and since $H^{d_0,*}_{\delta}$ is an outer measure, for $0<\delta<1$ we have
$H^{d_0,*}_\delta(N)\leq \sum_{U\in \mathscr{U}_\delta}H_\delta^{d_0}(U)\leq 
\sum_{U\in \mathscr{U}_\delta}H_\delta^{\dim(U)}(U)\leq \sum_{U\in \mathscr{U}_\delta}|U|^{\dim(U)}<\epsilon.$ Hence $H^{d_0}(N)\leq \epsilon.$  Since $\epsilon>0$ was arbitrary, $H^{d_0}(N)=0.$
\end{proof}

The following two propositions relate the pointwise dimension to the global dimension.

\begin{proposition} Suppose $X$ is separable with Hausdorff dimension $d_0.$ Let $A:=\dim_{\loc}^{-1}([0,d_0)).$ Then $H^{d_0}(A)=0.$ In particular, $H^{d_0}(X)=H^{d_0}(\dim_{\loc}^{-1}(\{d_0\}))$.
\end{proposition}
\begin{proof}
$A$ is open since $d$ is upper semicontinuous. For $x\in A$ let $U_x\in \mathscr{N}(x)$ with $\dim(U_x)<d_0$ and $U_x\subset A.$ Then the $U_x$ form an open cover of $A.$ Since $X$ has a countable basis, there exists a countable open cover $(U_k)$ of $A$ with the property that for all $x$ there exists a $k$ with $x\in U_k \subset U_x.$ In particular $\dim(U_k)\leq \dim(U_x)<d_0.$ Let $A_j:=\cup_{k\leq j} U_k.$ Then, since $H^s(A_j)\leq \sum_{k\leq j} H^s(U_k)$ for any $s\geq 0,$ $\dim(A_j)\leq \max_{k\leq j} \dim(U_k)<d_0.$  So $H^{d_0}(A_j)=0$ for all $j.$ But by continuity of measure, $H^{d_0}(A)=\sup_j H^{d_0}(A_j)=0.$ Since $d\leq d_0$ the other result follow immediately. 
\end{proof}

\begin{proposition}
Let $X$ be a separable metric space. Then $\dim(X)=\sup_{x\in X}\dim_{\loc}(x).$ Moreover, if $X$ is compact then the supremum is attained. 
\end{proposition}

\begin{proof}
Clearly $\sup_{x\in X}\dim_{\loc}(x) \leq \dim(X).$ Conversely, let $\epsilon>0.$ For $x\in X$ let $U_x\in \mathscr{N}(x)$ such that $\dim(U_x)\leq \dim_{\loc}(x)+\frac{\epsilon}{2}.$ Then the $(U_x)_{x\in X}$ form an open cover of $X.$ Since $X$ is separable it is Lindel\"{o}f. So let $(U_{x_i})_{i\in \mathbb{Z}_{+}}$ be a countable subcover. Then if $\sup_{i\in \mathbb{Z}_+}\dim(U_{x_i})=\infty$ then also $\dim(X)=\infty.$ Else if $\sup_{i\in \mathbb{Z}_+}\dim(U_{x_i})<t$ then $H^t(U_{x_i})=0$ for all $i$ and so $H^t(X)\leq \sum_{i=1}^\infty H^t(U_{x_i})=0.$ So $\dim(X)\leq t.$ Hence $\dim(X)=\sup_{i\in \mathbb{Z}_+}\dim(U_{x_i}).$ Then choose $j\in \mathbb{Z}_+$ such that $\dim(X)\leq \dim(U_{x_j})+\frac{\epsilon}{2}.$ Then $\dim(X)\leq \dim_{\loc}(x_j)+\epsilon\leq \sup_{x\in X}\dim_{\loc}(x).$ Hence $\dim(X)=\sup_{x\in X}\dim_{\loc}(x).$

Now suppose $X$ is compact. Let $d_0:=\dim(X).$ It remains to show that there exists some $x\in X$ with $\dim_{\loc}(x)=d_0.$ Suppose not. Then clearly $d_0>0.$ Let $m\geq 1$ such that $\frac{1}{m}<d_0.$ Then the sets $U_n:=\dim_{\loc}^{-1}[0,d_0-\frac{1}{n})$ for $n\geq m$ form an open cover of $X.$ By compactness there exists a finite subcover. So there exists an $N>0$ such that $X=\dim_{\loc}^{-1}[0,d_0-\frac{1}{N}).$ Hence $\sup_{x\in X} \dim_{\loc}(x)\leq d_0-\frac{1}{N}<d_0,$ a contradiction.

\end{proof}

\section{Variable Ahlfors $Q$-regularity}
Suppose $(X,d)$ is compact. For $Q:X\rightarrow [0,\infty)$ continuous, define $Q^-,Q^+, Q^c:\mathscr{B}\rightarrow [0,\infty)$ by $Q^-(U)=\inf_{x\in U} Q(x), Q^+(U)=\sup_{x\in U} Q(x).$ For arbitrary $Q:X\rightarrow [0,\infty)$ define $Q_c:\mathscr{B}\rightarrow 
[0,\infty)$ by $Q_c(B_r(x))=Q(x).$ Then for $\tilde{Q}:\mathscr{B}\rightarrow [0,\infty)$ with $Q^- \leq \tilde{Q}\leq Q^+,$ let $\lambda^{\tilde{Q}}:=\mu_\tau$, where $\tau$ is restricted to $\mathscr{B}$ and defined by $\tau(B):=|B|^{\tilde{Q}(U)}, \tau(\emptyset)=0.$

For $Q:X\rightarrow [0,\infty)$, we call $X$ $Q$-amenable if $0<\lambda^{Q_c}(B)<\infty$ for every non-empty open ball $B$ of finite radius in $X.$ In the case of $X$ compact this is equivalent to $\lambda^{Q_c}$ being finite with full support.

\begin{proposition} If $X$ is $Q$-amenable with $Q$ continuous then $\dim_{\loc}(x)=Q(x)$ for all $x\in X.$ 
\end{proposition}
\begin{proof} Let $B=B_r(x)$ a non-empty open ball, $q^-:=Q^-(B), q^+:=Q^+(B), d_0:=\dim(B).$ Let $B':=B_{\frac{r}{2}}(x)$ Then if $d_0<q^-, \dim(B')\leq d_0< q^-$ so $\lambda^{q^-}(B')=0.$ Let $0<\delta<\min\{\frac{r}{8}, 1\}.$ Let $\mathscr{U}\in \mathscr{B}_\delta(B')$ and $U\in \mathscr{U}.$ We may assume $U\cap B' \neq \emptyset$, since otherwise $\mathscr{U}$ may be improved by removing such a $U.$ Say $U=B_{r_U}(x_U).$
Moreover, we may assumue $r_U\leq 2\delta.$ Indeed, if $|U|=0$ and $r_U>\delta$ then $U=B_{\delta}(x_U)$ so we may take $r_U=\delta$ in that case. If $|U|>0$ then if $r_U>2|U|$ then $B_{r_U}(x_U) = B_{2|U|}(x_U)$ so we may take $r_U = 2|U|\leq 2\delta.$ Say $w\in U\cap B'.$ Then if $z \in U,$ 
$\rho(z,x)\leq \rho(z,x_U)+\rho(x_U,w)+\rho(w,x)\leq 2r_U+\frac{r}{2} \leq 4\delta+\frac{r}{2}<r.$ So $U\subset B.$ Hence $q^- \leq Q^c(U)$ and since $|U|\leq 1,$ $|U|^{q^-}\geq |U|^{Q^c(U)}$. So $\lambda^{Q_c}(B')=0,$ a contradiction. If $d_0>q^+$ then since $B_r(x)=\cup_{n=1}^\infty B_{r-\frac{1}{n}}(x),$ it is straightforward, using countable subadditivity of the measures $H^s$ and the definition of Hausdorff dimension, to verify that $\dim(B_r(x)) = \sup_{n\geq 1} \dim(B_{r-\frac{1}{n}}(x))$. Since $d_0>q^+,$ let $N$ so that $\dim(B_{r-\frac{1}{N}}(x))>q^+.$ Let $r'=r-\frac{1}{N}$ and $B'=B_{r'}(x).$ Then $\lambda^{q^+}(B') = \infty.$ Let $0<\delta<\frac{1}{4N}.$ Let $\mathscr{U}\in \mathscr{B}_\delta(B')$ and $U\in \mathscr{U}.$ We may assume $U\cap B' \neq \emptyset,$ say $w\in U\cap B'$. As before we may assume $r_U\leq 2\delta$ and so if $z\in U$ then $\rho(z,x)\leq 2r_u+r'\leq 4\delta+r'<r.$ So $U\subset B.$ Hence $Q_c(U)\leq q^+$. Since $|U|<1,$ $|U|^{q^+}\leq |U|^{Q_c(U)}.$  Hence $\lambda^{Q_c}(B')=\infty,$ a contradiction. Hence $Q^-(B)\leq \dim(B) \leq Q^+(B)$ for every non-empty open ball $B.$ 
The result then follows since $Q$ is continuous.  
\end{proof}

A Borel measure $\nu$ on a metric space $X$ is said to have local dimension $d_\nu(x)$ at $x$ if $\lim_{r\rightarrow 0^+} \frac{\log(\nu(B_r(x)))}{\log(r)}=d_\nu(x)$. Since the limit may not exist, we may also consider upper and lower local dimensions at $x$ by replacing the limit with an upper or lower limit, respectively. 

If $Q:X\rightarrow (0,\infty)$ is a bounded function, then a measure $\nu$ is called Ahlfors $Q$-regular if there exists a constant $C>0$ so that $\frac{1}{C} \nu(B_r(x))\leq  r^{Q(x)}\leq C\nu(B_r(x))$ for all $0<r\leq \diam(X)$ and $x\in X.$\cite{Sob} It can be immediately observed that such a measure $\nu$ is $Q$-amenable and has $d_\nu(x)=Q(x)$ for all $x$ .

A function $p$ on a metric space $(X,\rho)$ is log-H{\"o}lder continuous if 
there exists a $C>0$ such that $|p(x)-p(y)|\leq \frac{-C}{\log(\rho(x,y))}$ for all $x,y$ with $0<\rho(x,y)<\frac{1}{2}.$

\begin{lemma} For $X$ compact, if $Q:X\rightarrow (0,\infty)$ log-H{\"o}lder continuous, $\tilde{Q}:\mathscr{B}\rightarrow [0,\infty)$ with $Q^-\leq \tilde{Q} \leq Q^+,$ then $\lambda^{Q^+}\simeq \lambda^{Q^-} \simeq \lambda^{\tilde{Q}}.$
\end{lemma}
\begin{proof}
Let $U$ open with $0<|U|<\frac{1}{2}.$ 
Then for $x,y\in U,$ $|Q(x)-Q(y)|\leq \frac{-
C}{\log(|U|)}.$ Hence $0\leq 
\log(|U|)(Q^-(U)-Q^+(U))\leq C.$ Then $|U|^{Q^+(U)}\leq |U|^{Q^-(U)} \leq e^C|U|^{Q^+(U)}.$ The result follows.
\end{proof}
The following lemma may be found in \cite{Sob}.
\begin{lemma} If $\nu$ is Ahlfors $Q$-regular then $Q$ is log-H\"older continuous.\end{lemma}
\begin{proof} By Ahlfors regularity, there exist constants $C_1,C_2$ such that $\nu(B_r(x))\leq C_1r^{Q(x)},r^{Q(x)}\leq C_2\nu(B_r(x))$ for all $0<r\leq\diam(X),x\in X.$ Let $x,y\in X$ with $0<r:=\rho(x,y)<\frac{1}{2}.$ Say $Q(x)\geq Q(y).$ Since $Q$ is bounded, let $R<\infty$ be an upper bound for $Q$ and let $e^C:=C_1C_22^R.$ Then since $B_r(y)\subset B_{2r}(x),$ $r^{Q(y)}\leq C_2\nu(B_r(y))\leq C_2\nu(B_{2r}(x))\leq C_1C_2 2^R r^{Q(x)} = e^Cr^{Q(x)}.$ Hence $d(x,y)^{|Q(y)-Q(x)|}\geq e^{-C}.$ So $|Q(x)-Q(y)|\leq \frac{-C}{\log(\rho(x,y)}.$
\end{proof}

Hence, in particular, if $\nu$ is Ahlfors $Q$-regular then $Q$ is continuous.
A Borel measure $\nu$ is regular if for every Borel set $A$, 
\[\nu(A)=\sup\{\nu(F)\;|\;X \supset F \mbox{ closed }\}=\inf\{\nu(U)\;|\;A\subset U \mbox{ open }\}.\] We state the following classical result. We state the proof here for completeness, following \cite{van}.
\begin{lemma} A finite Borel measure on $X$ is regular. 
\end{lemma}
\begin{proof} Let $\mathscr{M}:=\{A\;|\; \sup\{\nu(F)\;|\;X \supset F \mbox{ closed }\}=\inf\{\nu(U)\;|\;A\subset U \mbox{ open }\}\}.$ Since $X$ is both open and closed, $X\in \mathscr{M}.$ Suppose $A\in \mathscr{M}.$ Then for $\epsilon>0$ if $F\subset A\subset U$ with $F$ closed, $U$ open, and $\nu(U\setminus F)<\epsilon,$ then $U^c\subset A^c \subset F^c,$ $U^c$ closed, $F^c$ open, and $\nu(F^c\setminus U^c)=\nu(U\setminus F)<\epsilon.$ It follows that $A^c \in \mathscr{M}.$ Let $(A_n)\subset \mathscr{M}.$ Then for $\epsilon>0$, for each $n$ let $F_n\subset A_n \subset U_n$ with $F_n$ closed, $U_n$ open, and $\nu(U_n\setminus F_n)<\frac{\epsilon}{2^{n+1}}.$ Then let $A=\cup A_n.$ Let $N$ large so that $\nu(\cup_{n=1}^N F_n)>\nu(\cup_n F_n)-\frac{\epsilon}{2}.$ Then let $F=\cup_{n=1}^N F_n, U=\cup U_n.$ Then $F\subset A \subset U,$ $F$ is closed, $U$ is open, and $\nu(U\setminus \cup_n F_n) = \nu(U)-\nu(\cup_n F_n)<\frac{\epsilon}{2}.$ So $\nu(U\setminus F)=\nu(U)-\nu(F)<\epsilon.$ So $A\in \mathscr{M}$. Hence $\mathscr{M}$ is a $\sigma-$algebra. Let $A\subset X$ closed. Then $\nu(A)=\sup\{\nu(F)\;|\;X \supset F \mbox{ closed }\}.$ Let $U_n=B_{\frac{1}{n}}(A).$ Then $A\subset U_n,$ each $U_n$ is open, and $\cap_n U_n = A.$ So $\inf_n \nu(U_n) = \nu(A)$ by continuity of measure, since $\nu(X)<\infty.$ Hence $\nu(A)\leq \inf\{\nu(U)\;|\;A\subset U \mbox{ open }\}\}\leq \inf_n \nu(U_n)=\nu(A).$ Hence $\mathscr{M}$ contains all Borel sets.  
\end{proof}

\begin{proposition}If $\nu$ is a finite Ahlfors $Q$-regular Borel measure on a separable metric space then $\nu\simeq \lambda^{Q_c}.$
\end{proposition}

\begin{proof} Since $Q$ is bounded, let $R>0$ be an upper bound for $Q$ and let $C_1,C_2$ constants such that $\nu(B_r(x))\leq C_1r^{Q(x)},r^{Q(x)}\leq C_2\nu(B_r(x))$for all $x\in X$ and $0<r\leq \diam(X).$ 
Let $A\subset X$ be Borel measurable. Then for $\delta>0$ let $(B_i)_{i\in I} \in 
\mathscr{B}_\delta(A).$ Say $B_i=B_{r_i}(x_i).$ Let 
$r_i^\prime=\sup\{\rho(x_i,y)\;|\;y\in B_i\}.$ Then $B_i\subset \{y\;|\;
\rho(x_i,y)\leq r_i^\prime\}=B_{r_i'}[x].$ Note, by Ahlfors regularity, $\nu(B_{r_i'+1}(x_i))<\infty$. So by continuity of measure $\nu(B_{r_i'}[x_i])=\inf_{n\geq 1} \nu(B_{r_i'+\frac{1}{n}}(x_i))\leq C_1\inf_{n\geq 1} (r_i'+\frac{1}{n})^{Q(x_i)}=C_1r_i'^{Q(x_i)}.$ It follows that $\nu(A)\leq \sum_{i\in I}\nu(B_i)\leq 
C_1\sum_i r_i'^{Q(x_i)} \leq C_1\sum_i |B_i|^{Q(x_i)}.$ Hence $\nu(A)\leq C_1\lambda_\delta^{Q_c}
(A)\leq C_1 \lambda^{Q_c}(A).$

Let $A$ be open. Let $\delta>0.$  Since $X$ is separable, let $(B_i)_{i\in I}\in \mathscr{B}_{\frac{\delta}
{10}}(A)$ such that $\cup B_i = A.$ Then by the 
Vitali Covering Lemma, there exists a disjoint subcollection $(B_j)_{j\in J}, J\subset I$ of the 
$(B_i)_{i\in I}$ such that $A\subset \cup_{j\in J} 5B_j.$ Let $Q_j=Q(x_j),$ where $x_j$ is 
the center of $B_j.$ Let $r_j$ be the radius of $B_j.$ Then $\lambda_\delta^{Q_c}(A)\leq 
\sum_j |5B_j|^{Q_j}\leq 10^R\sum_jr_j^{Q_j} \leq 10^R C_2\sum_j 
\nu(B_j)\leq 10^R C_2\nu(A).$ Since $\delta>0$ was arbitrary, $\lambda^{Q_c}(A)\leq C_2 10^R \nu(A).$
Let $B\subset X$ Borel measurable. Since $\nu$ is regular, for $\epsilon>0,$ let $A\subset X$ 
open with $B\subset A$ such that $\nu(B)\geq \nu(A)-\epsilon.$ Then $\lambda^{Q_c}(B)\leq \lambda^{Q_c}(A)\leq 10^RC_2(\nu(B)+\epsilon).$ Since $\epsilon>0$ is arbitrary, $\lambda^{Q_c}(B)\leq 10^RC_2\nu(B).$
\end{proof}
\begin{theorem} Let $X$ be a compact metric space. Then if $\nu$ is an Ahlfors $Q$-regular Borel measure then $Q=\dim_{\loc}$ and $\nu \simeq H_{\loc}.$ 
\end{theorem}

\begin{proof} Since $X$ is compact, $\nu$ is finite. Hence by the previous proposition $\nu\simeq \lambda^{Q_c}.$ Hence $X$ is $Q$ amenable and $Q$ is continuous, as it is log-H\"older continuous. So $Q=\dim_{\loc}.$ Let $B$ be a non-empty open ball. Then clearly $Q^-(B) \leq \dim(B).$ By continuity $Q^+(B)=Q^+(\bar{B}),$ where $\bar{B}$ is the closure of $B.$ By compactness, $\bar{B}$ is compact. Hence, as we have shown, $\dim(\bar{B})=\sup_{x\in \bar{B}} \dim_{\loc}(x) = Q^+(\bar{B})=Q^+(B).$ Therefore $Q^-\leq \dim(B) \leq Q^+$ on $\mathscr{B}.$  Hence, by log-H\"older continuity, $\lambda^{Q_c}\simeq \lambda_{\loc}.$ Finally, since $\lambda_{\loc}\simeq H_{\loc}$ and since $\simeq$ is transitive, we have  $\nu\simeq H_{\loc}.$
\end{proof}

Hence if a compact space admits an Ahlfors $Q$-regular Borel measure then that measure is strongly equivalent to the local measure.

\bibliographystyle{amsplain}
\bibliography{bib}

\end{document}